\documentclass[reqno,a4paper, 11pt]{amsart}

\usepackage[applemac]{inputenc}
\usepackage[T1]{fontenc}

\usepackage[english]{babel} 
\usepackage{amsmath}
\usepackage{amssymb}
\usepackage{amsthm}
\usepackage{color}
\usepackage{scrtime}
\usepackage{enumitem}
\usepackage{comment}
\usepackage{float}

\usepackage{microtype} 
\usepackage{graphicx}
\usepackage{tikz}
\usetikzlibrary{calc,fadings,decorations.pathreplacing,graphs}
\usepackage{tikz-3dplot}

\newcommand{\R}{\mathbb{R}}

\newcommand{\Z}{\mathbb{Z}}
\newcommand{\N}{\mathbb{N}}

\newcommand{\SP}{\mathbb{S}}

\newcommand{\Ker}{\operatorname{Ker}}

 \newtheorem{theorem}{Theorem}[section]
\newtheorem{lemma}[theorem]{Lemma}
\newtheorem{proposition}[theorem]{Proposition}

\theoremstyle{definition}
\newtheorem{definition}[theorem]{Definition}
\newtheorem{example}[theorem]{Example}
\newtheorem{remark}[theorem]{Remark}

\usepackage{color}
\usepackage{xstring}
\usepackage{ifthen}
\usepackage{siunitx}

\allowdisplaybreaks

\title{A family of sharp inequalities on real spheres}
\author{Roberto Bramati}
\address{Universit\'e de Lorraine, CNRS, IECL, F-57000 Metz, France} 
\email{roberto.bramati@univ-lorraine.fr}

\AtBeginDocument{%
   \def\MR#1{}
}

\begin{document}
\maketitle
\begin{abstract}
We prove a family of sharp multilinear integral inequalities on real spheres involving 
functions that possess some symmetries that can be described by annihilation by certain sets of vector fields. 
The Lebesgue exponents involved are seen to be related to the combinatorics of such sets 
of vector fields. Moreover we derive some Euclidean Brascamp--Lieb inequalities localized to a ball of radius $R$, with a 
blow-up factor of type $R^\delta$, where the exponent $\delta>0$ is related to the aforementioned 
Lebesgue exponents, and prove that in some cases $\delta$ is optimal. 
\end{abstract}

\section{Introduction}
Brascamp--Lieb inequalities on $\R^n$ are inequalities of the type
\begin{equation}\label{BLineq}
\int_{\R^n}\prod_{i=1}^m  f_i(B_i x)dx \leq C \prod_{i=1}^m \|f_i\|_{L^{p_i}(\R^{n_i})},
\end{equation}
where $B_i : \R^n\to\R^{n_i}$ are surjective linear maps, $f_i:\R^{n_i}\to\R^+$ are nonnegative measurable 
functions and $p_i\geq1$. The constant $C$, which depends on the maps $B_i$ and on 
the exponents $p_i$, is called Brascamp--Lieb constant and it is the 
smallest possible constant (finite or infinite) for which inequality \eqref{BLineq} holds for all 
nonnegative functions 
$f_i\in L^{p_i}(\R^{n_i})$. An effective tool to study these inequalities is the heat flow technique, introduced in this 
context by Carlen, Lieb and Loss \cite{CarlenLiebLoss04} and independently by Bennett, Carbery, Christ and Tao \cite{BCCT08}. For a review of the history of these inequalities, see \cite{BCCT08} and references therein.  

Our point of view on inequality \eqref{BLineq} is that it is an inequality for functions that possess symmetries described by 
annihilation by certain sets of vector fields. Indeed the functions $g_i=f_i\circ B_i$ are functions on $\R^n$ 
that are constant on affine subspaces parallel to the kernel of the map $B_i$. We can 
characterize such symmetry by saying that the functions $g_i$, which without loss of generality we suppose to be smooth, 
are annihilated by all vector fields  
parallel to the kernel of $B_i$. 
Inequality 
\eqref{BLineq} then says that for this kind of functions, the integral of the product is controlled, apart 
from a constant, by a product of Lebesgue norms of the functions restricted to a (linear) manifold where all the information is carried (we can see $\R^{n_i}$ as $(\Ker{B_i})^\perp$). 

A fundamental result contained in \cite{BCCT08} establishing necessary and sufficient conditions on the maps $B_i$ and the exponents $p_i$ in order for inequality \eqref{BLineq} to hold with finite constant in the Euclidean case is the following. 
\begin{theorem}\label{BCCT}
The constant $C$ in \eqref{BLineq} is finite if and only if $\sum_{i=1}^m p_i^{-1}n_i = n$ and 
for all $V$ subspaces of $\R^n$, $\dim(V)\leq \sum_{i=1}^m p_i^{-1}\dim(B_iV)$. 
\end{theorem}
The second condition in Theorem \ref{BCCT} can be interpreted as a non-degeneracy condition 
implying that the kernels of the maps $B_i$ are not too parallel between themselves, or that the symmetries 
involved are not too similar. 

The problem 
of finding {nonlinear} analogs of inequality \eqref{BLineq}  was first studied by  
Carlen, Lieb and Loss in \cite{CarlenLiebLoss04}, where the authors proved, using a heat flow argument, a sharp inequality 
on real spheres for functions depending on one variable. The problem was also studied from 
the more abstract point of view 
of Markov semigroups by Barthe, Cordero-Erausquin and Maurey in \cite{BCEM} and Barthe, Cordero-Erausquin, Ledoux and Maurey in \cite{BCELM}. 
Recently the author used the heat flow technique to prove some inequalities in the context 
of compact homogeneous spaces of Lie groups (see \cite{Bramati1}), providing some sharp results in the case of 
real spheres. The inequalities studied in \cite{Bramati1} have the form 
\begin{equation}\label{BLsphere}
\int_{\SP^{n-1}}\prod_{i=1}^m f_i d\sigma \leq \prod_{i=1}^m \|f_i\|_{L^{p_i}(\SP^{n-1})},
\end{equation}
where $d\sigma$ is the normalized uniform measure on $\SP^{n-1}$ and 
the functions $f_i$ have symmetries described by annihilation by certain differential operators 
and can be thought of as functions $\widetilde{f}_i$ defined on unit balls $B_m$ of 
Euclidean spaces of dimension $m\leq n-1$, then pulled-back to the sphere via the 
orthogonal projection $\pi_i$ from $\SP^{n-1}\subset\R^n$ onto the ball $B_m$. The $L^{p_i}(\SP^{n-1})$ norms 
of the functions $f_i$ can be controlled with the $L^{p_i}(B_m)$ norms of the $\widetilde{f}_i$, thus yielding a Brascamp--Lieb type inequality. 
Inequality \eqref{BLsphere} has also the structure of H\"older's inequality, but in the presence of symmetries 
certain exponents $p_i$ for which \eqref{BLsphere} holds could be not directly deducible from 
H\"older's inequality itself. Indeed, given exponents $p_i$ for which inequality \eqref{BLsphere} holds, H\"older's inequality implies that 
the inequality holds for $p'_i\geq  p_i$, but it is not clear how to get smaller exponents and what are the smallest possible (sharp) exponents {$\widetilde{p}_i$ for which the inequality holds for $p_i\geq\widetilde{p}_i$}.

In \cite{Bramati1} this sharp exponent is found for the cases of functions depending on $1\leq k  \leq n-2$ variables, and depending radially on $1\leq k  \leq n-2$ 
variables. These were the first sharpness results for this type of inequality, extending the sharpness found in  
\cite{CarlenLiebLoss04} for the case of functions depending on one variable. 

In this paper we determine sharp exponents for the inequality \eqref{BLsphere} when all symmetries of the functions 
$f_i$'s are balanced. { This means that the different Lie algebras $\mathcal{A}_i$ of vector fields annihilating the $f_i$'s are all isomorphic to each other and that all possible symmetries of the same type appear exactly once.} For example, we can consider the case of functions of 
$k$ variables. The balance condition requires that all the $f_i$'s are functions of $k$ variables, and all the possible collections of $k$ 
variables appear exactly once. 
We refer to Section 3 for the precise definition of balanced setting. In the balanced case, we find that the inequality holds if all exponents $p_i$ are bigger than or equal to a critical $\widetilde{p}\geq2$ independent of $i$.

Our main result is Theorem \ref{expsym} where we explicitly compute $\widetilde{p}$ and prove that this exponent is sharp, { meaning that it is the lowest possible for which the inequality holds for exponents $p_i\geq\widetilde{p}$ : for every $p<\widetilde{p}$ there are functions in $L^p(\SP^{n-1})$, with the appropriate symmetries, that make the left-hand side of \eqref{BLsphere} diverge.}
The computation of $\widetilde{p}$ is based on a general theorem proved by the author in \cite{Bramati1} that is recalled in Section 2.

As an application of Brascamp--Lieb type inequalities on spheres, we derive a family of local Brascamp--Lieb inequalities on Euclidean spaces, i.e. inequalities of type \eqref{BLineq} where on the 
left-hand side integration is performed just over a ball of radius $R>0$. 
Such local analogs of Brascamp--Lieb inequalities were first considered by Bennett, Carbery, Christ and Tao in \cite[Section 8]{BCCT08} 
and \cite{BCCT10}. More recently the growth rate in the parameter $R$  was studied in the case of weak Brascamp--Lieb 
inequalities (i.e. local inequalities with functions that are constant at certain scales) by Maldague \cite{Maldague19} and 
Zorin-Kranich \cite{ZK19}, with applications respectively to Multilinear Kakeya inequalities and 
Kakeya--Brascamp--Lieb inequalities. Our focus will be on the growth rate of the corresponding local Brascamp--Lieb constant, which will blow up as a power of $R$. In the balanced case, using the sharpness result of Theorem \ref{expsym}, we find that the exponent of $R$ is lowest possible.

\section{Notation and preliminary results}\label{notation}
Throughout this paper, for $A,B>0$, by $A\lesssim B$ we mean that $A\leq C B$, for some $C>0$.

We will interpret the unit sphere $\SP^{n-1}\subset \R^n$, for $n\geq3$, as the left homogeneous space 
$SO(n-1)\backslash SO(n)$. Let 
\[L_{i,j}=x_i\partial_{x_j}-x_j\partial_{x_i},\]
for $1\leq i< j \leq n$, be a basis for the Lie algebra $\mathfrak{so}(n)$ of left invariant vector fields on 
$SO(n)$, acting on $\SP^{n-1}$, and write the Laplace--Beltrami operator on $\SP^{n-1}$ as 
$L= \sum_{ i< j } L_{i,j}^2$. We say that a subset $\mathcal{A}\subseteq\{L_{i,j}\}_{i<j}$ is maximal if 
$\mathcal{A}= \langle\mathcal{A}\rangle\cap \{L_{i,j}\}_{i<j}$, where $\langle\mathcal{A}\rangle$ is the Lie subalgebra 
of $\mathfrak{so}(n)$ generated by $\mathcal{A}$. 
We denote by $\binom{a}{b}$, with $a\geq b\geq 0 $, the binomial coefficient and by
\[
\binom{a}{b_1,\dots,b_k}= \frac{a!}{b_1! \dots b_k!},
\]
for $k\in \N$ and $a=\sum_{i=1}^k b_i$, $b_i\geq 0$,  the multinomial coefficient. 

Let $\alpha=((\alpha)_1,\dots, (\alpha)_n)\in \Z_2^n:=(\Z/2\Z)^n=\{0,1\}^n$ be a multi-index, denote by $|\alpha|=\sum_{i=1}^n (\alpha)_i$ its length and by $\bar{\alpha}=(1,1,\dots,1)-\alpha$, where the 
difference is intended componentwise. We say that multi-indices $\alpha, \beta \in{ \Z_2^n}$ are orthogonal if $\alpha\cdot\beta=\sum_{i=1}^n 
(\alpha)_i(\beta)_i=0$, i.e. if they do not have 1's in the same components. 
For a point $x=(x_1,\dots,x_n)\in \R^n$, we denote by $x_\alpha$ the point 
with components $((\alpha)_1 x_1, \dots, (\alpha)_n x_n)$ and by $|x_\alpha|$ its Euclidean norm. Note that, by a 
small abuse of notation, the point $x_\alpha$ can be identified with a point in $\R^{|\alpha|}$.

Given $\alpha\in{ \Z_2^n}$ we denote by $\mathfrak{so}_{\alpha}$ the Lie algebra isomorphic to 
$\mathfrak{so}(|\alpha|)$ generated by the set $\{L_{k,l} : k<l, (\alpha)_k=(\alpha)_l=1\}$.
The following theorem holds.
\begin{theorem}[\cite{Bramati1}]\label{structure}
Let $\mathcal{A}\subseteq\{L_{i,j}\}_{i<j}$. Then there exist a unique $N\in\N$ and unique 
(up to relabeling in the case of equal length) pairwise orthogonal multi-indices $\alpha_1,\dots,\alpha_N$, 
with $|\alpha_i|\geq 2$, $|\alpha_1|\geq\dots\geq|\alpha_N|$,  and $\sum_{i=1}^N |\alpha_i|\leq n$, such that
\[
\langle\mathcal{A}\rangle = \bigoplus_{i=1}^N \mathfrak{so}_{\alpha_i }.
\]
\end{theorem}
\begin{remark}
If the subset $\mathcal{A}$ is maximal, its cardinality is necessarily $\sum_{i=1}^N \binom{|\alpha_i|}{2}$, and there is a natural 
splitting in $N$ subsets $\mathcal{A}_i$, of cardinality $\binom{|\alpha_i|}{2}$, each of which is a basis 
for the associated $\mathfrak{so}_{\alpha_i}$. 
\end{remark}
We are interested in subalgebras of the algebra of smooth functions on the sphere of functions which are 
annihilated by certain vector fields. In this regard we give the following definition. 
\begin{definition}\label{symmetric}
Let $\mathcal{A}\subseteq\{L_{i,j}\}_{i<j}$. A function $f\in C^\infty(\SP^{n-1})$ is $\mathcal{A}$-symmetric if 
$Xf=0$ for all $X\in \mathcal{A}$.
\end{definition}
\begin{remark}
A function which is $\mathcal{A}$-symmetric, is also $\mathcal{B}$-symmetric for subsets such that 
$\langle\mathcal{A}\rangle$=$\langle\mathcal{B}\rangle$, so it is convenient to consider only maximal subsets and we 
shall do so from now on. Note that a function which is $\mathcal{A}$-symmetric will also be annihilated by all vector 
fields in the Lie subalgebra $\langle\mathcal{A}\rangle$ of $\mathfrak{so}(n)$ and so it will be constant on certain submanifolds of $\SP^{n-1}$. 
\end{remark}
Given a multi-index $\alpha\in{ \Z_2^n}$, annihilation of a function $f(x_1,\dots,x_n)$ on the sphere by a subalgebra of 
type $\mathfrak{so}_\alpha$ gives radiality in the variables $x_\alpha$, i.e. the dependence on these variables 
is actually a dependence on $|x_\alpha|$.

For a function $f(x_1,\dots,x_n)$ on the sphere, annihilation by a maximal subset $\mathcal{A}$ and 
consequently by its generated subalgebra, which has the structure described in Theorem \ref{structure}, can be 
interpreted as follows. The multi-index $\alpha_1$ tells us that the function depends on the $n-|\alpha_1|$ variables 
$x_{\bar{\alpha}_1}$. Indeed $f(x_1, \dots, x_n) = f(x_{\alpha_1}, x_{\bar{\alpha}_1})$, being annihilated by 
$\mathfrak{so}_{\alpha_1}$, can be thought as a function 
$\widetilde{f}(|x_\alpha|, x_{\bar{\alpha}_1}) = \widetilde{f}(\pm\sqrt{1-|x_{\bar{\alpha}_1}|^2}, x_{\bar{\alpha}_1})$ which 
in turn can 
be identified with two functions $g_{\pm}(x_{\bar{\alpha}_1})$ defined on the ball 
$B_{n-|\alpha_1|}$ of $\R^{n-|\alpha_1|}$ (when we omit the indication of center and radius we refer to 
the unit ball centered at $0$). The ambiguity given by the $\pm$ sign is minor. 
Indeed one could split each function in the sum of two functions each defined on a different spherical cap 
and recover all the results that follow. To avoid heaviness of notation we will assume that all the functions we 
consider have an additional reflection symmetry, i.e., in the notation above, we require that $g_+(x_{\bar{\alpha}_1})=g_{-}(x_{\bar{\alpha}_1})$. 

The annihilation by the other subalgebras $\mathfrak{so}_{\alpha_i}$, for $i=2,\dots,N$, gives radial dependence on the collections 
of variables $x_{\alpha_i}$, which are contained in $x_{\bar\alpha_1}$ and distinct by the orthogonality of the multi-indices. 
The multi-index 
\begin{equation}\label{def_R}
R:=\bar{\alpha}_1 - \sum_{i=2}^N \alpha_i
\end{equation}
has 1's in the positions where the dependence is on the 
single variables. According to this interpretation, the ambiguity in Theorem \ref{structure} about the ordering of the multi-indices $\alpha_i$ gives rise to 
dependence on different variables (or radiality in different collections of variables), but, by the condition $x_1^2+\dots+x_n^2=1$, all these dependecies can be seen to be equivalent. See \cite{Bramati1} for further details about this interpretation and examples. 

We now recall the following theorem. 
\begin{theorem}[\cite{Bramati1}]\label{BLtheor}
Let $m\in \N$ and $\mathcal{A}^J$ be maximal subsets of $\{L_{i,j}\}_{i<j}$, for $J=1,\dots, m$. Then, for 
$\mathcal{A}^J$-symmetric nonnegative functions $f_J$, we have 
\begin{equation}\label{BLineq_2}
\int_{\SP^{n-1}} \prod_{J=1}^m f_J(x)d\sigma(x) \leq \prod_{J=1}^m \|f_J\|_{L^{p_J}(\SP^{n-1} )}
\end{equation}
for $p_J\geq \widetilde{p}$, where $\widetilde{p}$ is the number of occurrences of the 
most recurrent vector field among the finite sets $(\mathcal{A}^J)^c$, i.e. {
\[
\widetilde{p}= \max_{L\in \cup_{i=1}^m( \mathcal{A}^i)^c}
\,\, \max_{\substack{\boldsymbol{j}\, : \, \cap_{\boldsymbol{j}}(\mathcal{A}^k)^c \ni L}} |\boldsymbol{j}|,
\]
where $L$ denotes a vector field, $\boldsymbol{j}\in\{0,1\}^m$ and the notation $\bigcap_{\boldsymbol{j}} (\mathcal{A}^k)^c:= \bigcap_{i\,:\, (\boldsymbol{j})_k=1} (\mathcal{A}^k)^c$ denotes the intersection of the subsets $(\mathcal{A}^k)^c$ for those $k$ such that $(\boldsymbol{j})_k=1$.}
\end{theorem}	

\begin{remark}
Since $d\sigma(\SP^{n-1})=1$, by continuous embeddings of Lebesgue spaces on $\SP^{n-1}$, the relevant 
information of the Theorem is that inequality \eqref{BLineq_2} holds for $p_J=\widetilde{p}$ for all $J$. 
\end{remark}

\begin{remark}
Notice that by Theorem \ref{structure} all the information about the symmetries of the functions and 
the exponent $\widetilde{p}$ is contained in the multi-indices $\alpha^J_i$, for 
$i=1,\dots, N_J$, with $J=1,\dots,m$. 
\end{remark}

Theorem \ref{BLtheor} provides the same exponent $\widetilde{p}$ for all the functions. A more careful analysis leads to the following 
theorem.

\begin{theorem}[\cite{Bramati1}]\label{BLtheor_1}
With the hypotheses above, the inequality 
\begin{equation}\label{BLineq_1}
\int_{\SP^{n-1}} \prod_{J=1}^m f_J(x)d\sigma(x) \leq \prod_{J=1}^m \|f_J\|_{L^{p_J}(\SP^{n-1} )}
\end{equation}
holds for $p_J\geq \widetilde{p}_J$, where $\widetilde{p}_J$ is the number of occurrences of the 
most recurrent vector field of $(\mathcal{A}^J)^c$ among the finite sets $(\mathcal{A}^k)^c$, i.e. {
\[
\widetilde{p}_J= \max_{L\in( \mathcal{A}^J)^c}
\,\, \max_{\substack{\boldsymbol{j}\, : \, \cap_{\boldsymbol{j}}(\mathcal{A}^k)^c \ni L}} |\boldsymbol{j}|,
\]
where $L$ denotes a vector field, $\boldsymbol{j}\in\{0,1\}^m$ and the notation $\bigcap_{\boldsymbol{j}} (\mathcal{A}^k)^c:= \bigcap_{i\,:\, (\boldsymbol{j})_k=1} (\mathcal{A}^k)^c$ is as explained above.}
\end{theorem}
	
\begin{remark}
Notice that, by their definitions, $\widetilde{p}_J\leq\widetilde{p}$, so that \eqref{BLineq_1} is actually an improvement of \eqref{BLineq_2}. 
The problem of the sharpness of the exponents $p_J$ in Theorem \ref{BLtheor_1}, i.e. whether they are the smallest possible for which the Theorem holds, is open in the general case. 
Nevertheless for some classes of functions, like those treated in this paper, we can prove that they are sharp (see also Remark \ref{sharpness_rem} below). 
\end{remark}

{\section{Inequalities in the balanced case}\label{ineqSection}}
We are interested in the case where all the functions have the same type of symmetry (i.e. the Lie 
subalgebras generated by the maximal subsets for which they are symmetric are isomorphic) and we consider all the possible symmetries of the same type. Such balance allows to treat easily 
the combinatorics. In the unbalanced case Theorems \ref{BLtheor} and \ref{BLtheor_1} obviously still 
apply but it seems much harder to find an explicit form for the exponents and to prove that they are sharp.

Let us describe, using the notation introduced above, what we mean by balanced set up. We are in a balanced case if $N_J=N$ for all $J$, for some fixed $N\in\N$, and the 
multi-indices $\alpha^J_i$ have the same length 
$\widetilde{\alpha}_i$ for all $J$. In other words, we fix $N$ natural numbers $\widetilde{\alpha}_1,\dots,\widetilde{\alpha}_N$, with 
$\widetilde{\alpha}_i\geq 2$ and $\widetilde{\alpha}_1\geq\widetilde{\alpha}_2\geq\dots\geq\widetilde{\alpha}_N$ and 
$\sum_{i=1}^N \widetilde{\alpha}_i \leq n$. Moreover, in the balanced case, we also require to consider {exactly once} all the possible $N$-tuples of pairwise 
orthogonal multi-indices with those fixed 
lengths. We set $\widetilde{R}= n- \sum_{i=1}^N \widetilde{\alpha}_i$, which 
will be the cardinality of all multi-indices $R^J$, as defined in \eqref{def_R}. 
To each $N$-tuple we associate a maximal subset and thus a symmetry as described in 
Theorem \ref{structure}. The Lie subalgebras of $\mathfrak{so}(n)$ generated by this choice of multi-indices are obviously isomorphic and exhaust their isomorphism class (within our restricted choice of generators, which must be contained in the basis
 $\{L_{i,j}\}_{i<j}$). If $\widetilde{\alpha}_i=\widetilde{\alpha}_j$ for some $1\leq i < j \leq N$, in each 
$N$-tuple of multi-indices there will be an ambiguity in the ordering as pointed out in Theorem \ref{structure}. For the 
moment we will not care about this (see Remark \ref{card}) and count each case as separate. 

The problem thus becomes a problem about the combinatorics of the multi-indices. It is easy to see that the number of possible $N$-tuples of pairwise orthogonal multi-indices with fixed lengths $\widetilde\alpha_1,\dots,\widetilde\alpha_N$ is 
\[ 
\binom{n}{\widetilde\alpha_1,\dots,\widetilde\alpha_N,\widetilde R}:=J_{\max}
\]
so that $J$ runs from $1$ to $J_{\max}$. This is also the number of maximal subsets $\mathcal{A}^J$ in the balanced case and the 
number of functions, each $\mathcal{A}^J$-symmetric for a different $J$, that we will consider. 

\begin{theorem}\label{expsym}
In the balanced case, the exponent in inequality \eqref{BLineq_2} given by Theorem \ref{BLtheor} is 
\begin{equation}\label{exponent}
\widetilde{p}=
 \frac{(n-2)!\left(n(n-1)-\sum_{i=1}^N \widetilde\alpha_i(\widetilde\alpha_i-1)\right)}{\widetilde\alpha_1 ! \cdot\dots\cdot\widetilde\alpha_N ! \widetilde R !}.
\end{equation}
Moreover this exponent is sharp, in the sense that it cannot be lowered (see remark below). 
\end{theorem}

The proof of the sharpness is postponed to Section \ref{sharpSection}. 

\begin{remark}\label{sharpness_rem}
Following \cite{CarlenLiebLoss04} and the terminology established in \cite{Bramati1}, by sharp exponent we mean that the exponent \eqref{exponent} is the lowest possible for which Theorem \ref{BLtheor} holds as stated. The conditions given by Theorem \ref{BLtheor} are only sufficient and not necessary for the inequality to hold. Necessary conditions on the exponents for these inequalities are not known, even in the easiest cases (see \cite[Section 6.3]{Bramati1} for the case of functions of one variable on $\SP^2$). Finding them is an interesting and non-trivial open problem for future investigation. If we are in the case of Theorem \ref{BLtheor_1}, i.e. when all the exponents are the same, the condition given there is also necessary, if we are in the balanced case.  
\end{remark}

\begin{remark}
In this setting, the { critical} exponents given by Theorems \ref{BLtheor} and \ref{BLtheor_1} coincide. 
\end{remark}

\begin{proof}[Proof of Theorem \ref{expsym}]
We want to apply Theorem \ref{BLtheor}. {Notice that by the balance conditions, all vector fields $\{L_{i,j}\}_{i<j}$ will appear in $\bigcup_{J=1}^m (\mathcal{A}^J)^c$, and they will all have the same number of occurrences. This is easily seen by considering the  multi-indices associated to the maximal subsets. The balanced set up ensures that no vector field has a prominent role among the others. This also implies that the outer maximums in the definition of the exponents $\widetilde{p}$, $\widetilde{p}_J$  in Theorems \ref{BLtheor} and \ref{BLtheor_1} are irrelevant, meaning that they return the same value for all their inputs, yielding $\widetilde{p}_j=\widetilde{p}$ for all $J$. }
So we can just fix a vector field $L_{i,j}$ and count 
how many $\mathcal{A}^J$ contain it. The vector field $L_{i,j}$ is contained in $\mathcal{A}^J$ if 
$(\alpha_k^J)_i=(\alpha_k^J)_j=1$ for some $k=1,\dots,N$, and $1\leq i < j \leq n$. The number of ways of 
doing this is 
\[
\sum_{i=1}^N \binom{n-2}{\widetilde\alpha_1,\dots,\widetilde\alpha_i - 2,\dots, \widetilde\alpha_N, \widetilde R}.
\]
So our $\widetilde{p}$ is given by 
\[
\binom{n}{\widetilde\alpha_1,\dots,\widetilde\alpha_N,\widetilde R} - 
\sum_{i=1}^N \binom{n-2}{\widetilde\alpha_1,\dots,\widetilde\alpha_i - 2,\dots, \widetilde\alpha_N, \widetilde R},
\]
which after easy manipulation gives \eqref{exponent}.
\end{proof}

\begin{example}
The simplest balanced case $(N=1)$ is that of the functions of $k$ variables, for $1\leq k\leq n-2$, 
which was treated in \cite{Bramati1} (see also 
\cite{CarlenLiebLoss04}, where the case $k=1$ was first established, and \cite{BCELM}). 
By Theorem \ref{expsym} we have 
$\widetilde\alpha_1=n-k$, $J_{\max}=\binom{n}{\widetilde\alpha_1}$ and
\[
\widetilde{p}= \binom{n}{\widetilde\alpha_1}- \binom{n-2}{\widetilde\alpha_1 -2} = \binom{n}{k}-\binom{n-2}{k}. 
\]
\end{example}
\begin{example}
The case $N=2$, $\widetilde R=0$ is that of functions depending radially on 
$1\leq k \leq n-1$ variables. By Theorem \ref{expsym} we have 
$\widetilde\alpha_1= \max\{n-k, k\}$, and $\widetilde\alpha_2=\min\{n-k,k\}$, $J_{\max}=\binom{n}{\widetilde\alpha_1,\widetilde\alpha_2}$ 
and 
\[
\widetilde{p}=
 \frac{(n-2)!\left(n(n-1)- (n-k)(n-k+1) - k(k+1)\right)}{(n-k) ! k!} = 2\binom{n-2}{k-1}.
\]
For both examples, the author proved in \cite{Bramati1} that the exponents are sharp.
\end{example}
In the following section we will prove that all exponents coming from Theorem \ref{expsym} are sharp. 

\section{Sharpness}\label{sharpSection}
For a fixed symmetry type, i.e. for fixed lengths $\widetilde\alpha_1,\dots,\widetilde\alpha_N$ as above, we consider the 
possible maximal subsets $\mathcal{A}^J$, for $J=1,\dots,J_{\max}$ and introduce the following 
functions
\begin{align}\label{functions}
f_J(|x_{\alpha_2^J}|, \dots, &|x_{\alpha_N^J}|, x_{R^J})=
\prod_{i=2}^N |x_{\alpha_i^J}|^{-\gamma\widetilde\alpha_i}\prod_{i : (R^J)_i = 1} |x_i|^{-\gamma} \nonumber\\
&+\sum_{i=2}^N (1-|x_{\alpha_i^J}|^2)^{-\frac{\gamma(n-\widetilde\alpha_i)}{2}} + \sum_{i : (R^J)_i = 1} (1-x_i^2)^{-\frac{\gamma(n-1)}{2}}.
\end{align}
with $\gamma>0$ to be determined. Note that the function $f_J$ is $\mathcal{A}^J$-symmetric. 
In order to estimate Lebesgue norms of these functions we will use the following lemma. 
\begin{lemma}[\cite{Grafakos, StW}]\label{integration}
Let $\alpha\in{ \Z_2^n}$ and $f(x_\alpha)$ a $\mathfrak{so}_{\bar\alpha}$-symmetric function, i.e. a function depending 
on $|\alpha|$ variables. Then 
\begin{equation}\label{int_formula}
\int_{\SP^{n-1}} f(x_\alpha)d\sigma{= C_{n,|\alpha|}}\int_{B_{|\alpha|}}f(x_{\alpha}) (1-|x_\alpha|^2)^\frac{n-2-|\alpha|}{2} dx_\alpha, 
\end{equation}
{where $C_{n,|\alpha|}$ is a constant only depending on $n$ and $|\alpha|$.}
\end{lemma}
With the integration formula provided by Lemma \ref{integration} we can prove the following proposition. 
\begin{proposition}\label{lp_norms}
For functions $f_J$ as in \eqref{functions} we have $\|f_J\|_{L^p(\SP^{n-1})}<\infty$ for $\gamma p < 1$. 
\end{proposition}
\begin{proof}
By convexity we have 
\begin{align*}
\|f_J\|_{L^p(\SP^{n-1})}^p &\lesssim \int_{\SP^{n-1}}\prod_{i=2}^N |x_{\alpha_i^J}|^{-\gamma\widetilde\alpha_i p}\prod_{i : (R^J)_i = 1} |x_i|^{-\gamma p} d\sigma \\
&+\sum_{i=2}^N \int_{\SP^{n-1}}(1-|x_{\alpha_i^J}|^2)^{-\frac{\gamma(n-\widetilde\alpha_i) p }{2}} d\sigma + \sum_{i : (R^J)_i = 1}\int_{\SP^{n-1}}(1- x_i^2)^{-\frac{\gamma(n-1)p}{2}}d\sigma\\& = I + \sum II_i + \sum III_i .
\end{align*}
For the term $I$, since the integrand is a function of $n-\widetilde\alpha_1$ variables, we use \eqref{int_formula} and pass to polar 
coordinates to get 
\begin{align*}
I &\lesssim  \int_{B_{n-\widetilde\alpha_1}}\prod_{i=2}^N |x_{\alpha_i^J}|^{-\gamma\widetilde\alpha_i p}\prod_{i : (R^J)_i = 1} |x_i|^{-\gamma p} (1-|x_{\bar{\alpha_1}^J}|^2)^{\frac{\widetilde\alpha_1-2}{2}}dx_{\bar\alpha_1^J} \\
&\lesssim \prod_{i=2}^N \int_{B_{\widetilde\alpha_i}} |x_{\alpha_i^J}|^{-\gamma\widetilde\alpha_i p} dx_{\alpha_i^J} \prod_{i : (R^J)_i = 1} 
\int_{-1}^1 |x_i|^{-\gamma p} dx_i  \\
&\lesssim \prod_{i=2}^N \int_0^1 \rho^{-\gamma\widetilde\alpha_i p+\widetilde\alpha_i-1} d\rho \prod_{i : (R^J)_i = 1} 
\int_{-1}^1 |x_i|^{-\gamma p} dx_i.
\end{align*}
The integrals are finite if $-\gamma\widetilde\alpha_i p+\widetilde\alpha_i-1>-1$ and $-\gamma p > -1$, i.e. when $\gamma p <1$. 
For each of the pieces $II_i$ we use again \eqref{int_formula} and polar coordinates to get 
\begin{align*}
II_i &\lesssim  \int_{B_{\widetilde\alpha_i}}(1-|x_{\alpha_i^J}|^2)^{-\frac{\gamma(n-\widetilde\alpha_i) p }{2}} (1-|x_{\alpha_i^J}|^2)^\frac{n-2-
\widetilde\alpha_i}{2} dx_{\alpha_i^J} \\
&\lesssim \int_0^1 (1-\rho^2)^{-\frac{\gamma(n-\widetilde\alpha_i) p }{2} + \frac{n-2-
\widetilde\alpha_i}{2}}\rho^{\widetilde\alpha_i -1 }d\rho. 
\end{align*}
The integral is finite if $ -\frac{\gamma(n-\widetilde\alpha_i) p }{2} + \frac{n-2-
\widetilde\alpha_i}{2} > -1$, which again gives $\gamma p < 1$. 
The same computation works for the terms of type $III_i$, which involve functions of one variable, and it is easily seen that 
the integrability condition is again 
$\gamma p<1$. 
\end{proof}

We are now ready to prove the following theorem.
\begin{theorem}\label{sharp_exp}
The exponent $\widetilde{p}$ in Theorem \ref{expsym} is sharp, i.e. { for each $p<\widetilde{p}$ the Theorem \ref{expsym} does not hold, since} there exist functions $f_J$, each 
$\mathcal{A}^J$-symmetric, for $J=1,\dots,J_{\max}$, such that the right-hand side of \eqref{BLineq_2} is finite and the 
left-hand side diverges. 
\end{theorem}
\begin{proof}
By Proposition \ref{lp_norms}, in order to have a finite right-hand side in \eqref{BLineq_2} it suffices to have $\gamma p < 1$.

{ Let us 
now consider the left-hand side of \eqref{BLineq_2}, that we want to control from below. 
In the left-hand side, the product inside the integral gives a sum of (positive) products of the terms of the considered functions. We dominate all terms of this sum from below with just a specific one:} we fix a variable, say $x_n$, and we select the 
summand that contains the product term (the first term in \eqref{functions}) for those $J$ for which the function $f_J$ does not depend on the variable $x_n$, 
i.e. if $(\alpha_1^J)_n=1$, and the sum term that contains $x_n$ (either the second or the third term in \eqref{functions}), for those $J$ for 
which $f_J$ depends on $x_n$, i.e. if $(\alpha_1^J)_n=0$.
The number of functions that do not depend on the variable $x_n$ is 
$
\binom{n-1}{\widetilde\alpha_1-1, \dots, \widetilde\alpha_N, \widetilde R}.
$
The number of functions that depend on the variable $x_n$ in the radial collection $|x_{\alpha^J_i}|$ for $i=2,\dots, N$ is 
$
\binom{n-1}{\widetilde\alpha_1, \dots, \widetilde\alpha_i-1,\dots, \widetilde\alpha_N, \widetilde R}, 
$
and for these functions we denote by $\hat i $ the unique index $i$ such that $(\alpha^J_{\hat i})=1$. 
Finally the number of functions that depend on the variable $x_n$ as a single variable, that are the functions such 
that $(R^J)_n=1$, is 
$
\frac{\widetilde R}{n}\binom{n}{\widetilde\alpha_1,\dots,\widetilde\alpha_N,\widetilde R},
$ and the expression also includes the case $\widetilde R=0$. Note that 
\[
\sum_{i=1}^N \binom{n-1}{\widetilde\alpha_1, \dots, \widetilde\alpha_i-1,\dots, \widetilde\alpha_N, \widetilde R} + \frac{\widetilde R}{n}\binom{n}{\widetilde\alpha_1,\dots,\widetilde\alpha_N,\widetilde R} = \binom{n}{\widetilde\alpha_1,\dots,\widetilde\alpha_N,\widetilde R}, 
\]
which is $J_{\max}$, the total number of functions involved. For $\beta>0$ we will use the inequalities $|x_\alpha|^{-\beta} \geq (x_1^2+\dots x_{n-1}^2)^{-\frac{\beta}{2}}$ for 
collections 
$x_\alpha$ that 
do not contain the variable $x_n$, and $(1-|x_\alpha|^2)^{-\frac{\beta}{2}}\geq (1-x_n^2)^{-\frac{\beta}{2}}$ for 
collections $x_\alpha$ that contain the $x_n$ variable. 
We have 
\begin{align*}
\int_{\SP^{n-1}} &\prod_{J=1}^{J_{\max}} f_J \, d\sigma \geq \int_{\SP^{n-1}} \prod_{J : (\alpha^J_1)_n=1} \left(\prod_{i=2}^N |x_{\alpha_i^J}|^{-\gamma\widetilde\alpha_i }\prod_{i : (R^J)_i = 1} |x_i|^{-\gamma} \right) \\
&\times
\prod_{J : (\alpha^J_1)_n=0, (R^J)_n=0} ( 1- |x_{\alpha^J_{\hat i}}|^2)^{-\frac{\gamma(n-\widetilde\alpha_{\hat i})}{2} }\prod_{J: (R^J)_n=1} (1-|x_n|^2)^{-\frac{\gamma(n-1)}{2}}d\sigma\\
&\geq \int_{\SP^{n-1}} \prod_{J : (\alpha^J_1)_n=1} (x_1^2+\dots+x_{n-1}^2)^{-\frac{\gamma\left(\sum_{i=2}^N\widetilde\alpha_i + \widetilde R\right)}{2}} \\
&\times
\prod_{J : (\alpha^J_1)_n=0, (R^J)_n=0} (1- x_n^2)^{-\frac{\gamma(n-\widetilde\alpha_{\hat i})}{2}} \prod_{J: (R^J)_n=1} (1-x_n^2)^{-\frac{\gamma(n-1)}{2}}d\sigma.
\end{align*}
Now we think of the sphere as a graph, noting that $1-x_n^2=x_1^2+\dots+x_{n-1}^2$, and pass to polar coordinates, obtaining 
\begin{align*}
\int_{\SP^{n-1}} &\prod_{J=1}^{J_{\max}} f_J \, d\sigma \geq \int_{B_{n-1}} \prod_{J : (\alpha^J_1)_n=1} (x_1^2+\dots+x_{n-1}^2)^{-\frac{\gamma\left(\sum_{i=2}^N\widetilde\alpha_i + \widetilde R\right)}{2}} \\
&\times
\prod_{J : (\alpha^J_1)_n=0, (R^J)_n=0}(x_1^2+\dots+x_{n-1}^2)^{-\frac{\gamma(n-\widetilde\alpha_{\hat i})}{2}} \\
&\times\prod_{J: (R^J)_n=1} (x_1^2+\dots+x_{n-1}^2)^{-\frac{\gamma(n-1)}{2}}(1-x_1^2-\dots-x_{n-1}^2)^{-1/2}{dx_1\dots dx_{n-1}}\\
&= \int_0^1\rho^{-\gamma\left[\left(\sum_{i=2}^N\widetilde\alpha_i + \widetilde R\right)\binom{n}{\widetilde\alpha_1-1, \dots, \widetilde\alpha_N, \widetilde R} + \sum_{i=2}^N (n-\widetilde\alpha_i)\binom{n-1}{\widetilde\alpha_1, \dots, \widetilde\alpha_i-1,\dots, \widetilde\alpha_N, \widetilde R}\right] }\\
&\times \rho^{-\gamma\left[\frac{\widetilde R (n-1)}{n}\binom{n}{\widetilde\alpha_1,\dots,\widetilde\alpha_N,\widetilde R}\right]} {(1-\rho^2)^{-1/2}}{\rho^{n-2}d\rho}.
\end{align*}
This integral diverges for 
\begin{align*}
\gamma &= (n-1) \left[ \left(\sum_{i=2}^N\widetilde\alpha_i + \widetilde R\right)\binom{n}{\widetilde\alpha_1-1, \dots, \widetilde\alpha_N, \widetilde R} 
\right. \\
& \left. + \sum_{i=2}^N (n-\widetilde\alpha_i)\binom{n-1}{\widetilde\alpha_1, \dots, \widetilde\alpha_i-1,\dots, \widetilde\alpha_N, \widetilde R} + \frac{\widetilde R (n-1)}{n}\binom{n}{\widetilde\alpha_1,\dots,\widetilde\alpha_N,\widetilde R}  \right]^{-1}.
\end{align*}
So, if we take $\gamma p <1$ to make the right-hand side finite, we find
\[
p < \frac{1}{\gamma} = \frac{(n-2)!}{\widetilde\alpha_1!\dots\widetilde\alpha_N!\widetilde R!}\left[\sum_{i=1}^N (n-\widetilde\alpha_i)\widetilde\alpha_i +(n-1)\widetilde R \right],
\]
which is exactly $\widetilde{p}$ in Theorem \ref{expsym}. 
\end{proof}

\begin{remark}\label{card}
There is a potential ambiguity in the ordering of the multi-indices in Theorem 
\ref{structure}, but the symmetries related to different orderings are equivalent, meaning that the Lie subalgebras of $\mathfrak{so}(n)$ given by Theorem \ref{structure} associated to the ambiguous cases are not only isomorphic, but exactly the same subalgebra. This 
also implies that in Theorems \ref{expsym} and \ref{sharp_exp} we might be over-counting the 
number of functions $J_{\max}$ and all the other related quantities. Indeed, there is a common 
factor multiplying all these quantities, due to the fact that we introduced the ordering in Theorem \ref{structure}. 
To identify this factor, let $A_j$, $j\in\N$ be the set of indices $\widetilde\alpha_i$ such that $\widetilde\alpha_i=j$, for $i=1,\dots,N$. 
By our construction, the only sets that can be nonempty are $A_j$ for $j=2,\dots,n-2$. 
We can run the proofs of Theorems \ref{expsym} and \ref{sharp_exp} with all the 
quantities divided by $\prod_{j=2}^{n-2}|A_j|!$ and get again sharp exponents.
\end{remark}

\section{Local Brascamp--Lieb inequalities}
As an application of the inequalities found in Section \ref{ineqSection} we derive some local Euclidean Brascamp--Lieb inequalities 
associated to orthogonal projections on vector subspaces of $\R^n$ generated by collections of vectors in the basis $\{e_1,\dots,e_n\}$, i.e. inequalities 
on $\R^n$ of the form 
\[
\int_{B(0,R)} \prod_{j=1}^m f_j(\pi_j x) dx \lesssim R^\delta \prod_{j=1}^m{\|f_j\|_{L^p(\R^{n_j})}},
\]
where $B(0,R)$ is the Euclidean ball of center $0$ and radius $R>0$ (this $R$ is not related to the $R$ appearing in the previous sections) and the power $\delta$ and the implicit constant depend on $n$, on the projections $\pi_j$ 
and on the exponent $p$. More precisely we only consider projections $\pi_\alpha : \R^n \to \R^{|\alpha|}$ mapping a 
point $(x_1,\dots,x_n)$ to $x_\alpha$ (see Section \ref{notation} for the notation). Note that, given a function $f : \R^{|\alpha|}\to \R^+$, 
the pullback function $f\circ \pi_\alpha : \R^n\to \R^+$ is a function that restricted to each sphere $r\SP^{n-1}$, with $r>0$, endowed 
with the normalized measure $r^{-(n-1)}d\sigma$, with $d\sigma$ as above, is annihilated by the algebra $\mathfrak{so}_{\bar\alpha}$. 
We thus can extend Definition \ref{symmetric} to functions defined on the whole space $\R^n$. Recall that for a maximal subset $\mathcal{A}$ of $\{L_{i,j}\}_{i,j}$, we have the decomposition 
\[
\langle \mathcal{A}\rangle  = \bigoplus_{i=1}^N\mathfrak{so}_{\alpha_i}
\]
given by Theorem \ref{structure}. A function $f$ on the whole space $\R^n$ which is $\mathcal{A}$-symmetric is a function that 
depends only on the variables $x_{\bar{\alpha}_1}$, { and on the radius $|x_{\alpha_1}|$. Hence, $\mathcal{A}$-symmetric functions on $\R^n$ are more general than those appearing in Euclidean Brascamp--Lieb inequalities associated to orthogonal projections. Anyway, in order to analyze the connection with the Euclidean case, for the sake of our analysis, we will only consider $\mathcal{A}$-symmetric functions on $\R^n$ that do not depend on the radius $|x_{\alpha_1}|$. Since they only depend on the variables $x_{\bar{\alpha}_1}$, they can be identified 
with a function $\widetilde{f}$ defined on $\R^{n-|\alpha_1|}$ and then pulled back via the projection $\pi_{\bar{\alpha}_1}$.} Notice that if we restrict a function $f$ of this kind to a sphere $r\SP^{n-1}$ of radius $r$, we have 
\[
(\widetilde{f}\circ\pi_{\bar{\alpha}_1})_{|_{r\SP^{n-1}}} = \widetilde{f} \circ (\pi_{\bar{\alpha}_1})_{|_{r\SP^{n-1}}}, 
\] 
and $\widetilde{f}$ now acts only on the ball $rB_{n-|\alpha_1|}$ of radius $r$ in $\R^{n-|\alpha_1|}$ and so we recover the interpretation 
of the previous section. {By a slight abuse of notation we will write $f(x_{\bar{\alpha}_1})$ for $\widetilde{f}(x_{\bar{\alpha}_1})$}.

The following theorem is a special case of the local Brascamp--Lieb inequalities introduced by Bennett, Carbery, Christ, and Tao and is an immediate consequence of their analysis (see \cite[Theorem 8.17]{BCCT08} and \cite[Theorem 2.2]{BCCT10}). Here we propose a different proof for our particular case, using a change of variables to spherical coordinates and applying Theorem \ref{BLtheor_1}.

\begin{theorem}\label{local_th}
Let $f_1,\dots,f_m$ functions on $\R^n$ and let each $f_J$ be $\mathcal{A}^J$-symmetric for some maximal subset $\mathcal{A}^J$ of 
$\{L_{i,j}\}_{i<j}$, for $J=1,\dots,m$, {each of the form $f_J(x_{\bar{\alpha}_1^J})$}. Then the inequality  
\begin{equation}\label{localBL}
\int_{B(0,R)} \prod_{J=1}^m f_J({x_{\bar{\alpha}_1^J}}) dx \lesssim R^{\widetilde\delta} \prod_{J=1}^m{\|f_J\|_{L^{\widetilde{p}_J}(\R^{n-|\alpha_1^J|})}}
\end{equation}
holds for $R>0$, with $\widetilde{p}_J$ as in Theorem \ref{BLtheor_1} and 
\[
\widetilde\delta= n-\sum_{J=1}^m\widetilde{p}_J^{-1}(n-|\alpha^J_1|)>0.
\] 
\end{theorem}
\begin{proof}
Integrating the product $\prod_{J=1}^m f_J$ over the ball $B(0,R)$, passing to spherical coordinates and applying 
Theorem \ref{BLtheor_1}, we get
\begin{align*}
\int_{B(0,R)} \prod_{J=1}^m f_J(x) dx & = \int_0^R \int_{\SP^{n-1}} \prod_{J=1}^{m} f_J(\rho x') d\sigma\rho^{n-1}d\rho \\
&\lesssim \int_0^R \prod_{J=1}^{m} \| f_J(\rho\cdot)\|_{L^{\widetilde{p}_J}(B_{n-|\alpha^J_1|})}\rho^{n-1}d\rho \\
&=  \int_0^R \prod_{J=1}^{m} \| f_J \|_{L^{\widetilde{p}_J}(B_{n-|\alpha^J_1|}(0,\rho))}\rho^{-\sum (n-|\alpha^J_1|)\widetilde{p}_J^{-1}}\rho^{n-1}d\rho\\
&\lesssim \prod_{J=1}^{m} \| f_J \|_{L^{	\widetilde{p}_J}(\R^{n-|\alpha_1^J|})}\int_0^R\rho^{-\sum (n-|\alpha^J_1|)\widetilde{p}_J^{-1}+ n-1}d\rho.
\end{align*}
Observing that \[n-1 -\sum_{J=1}^m (n-|\alpha_1^J|)\widetilde{p}_J^{-1} \geq n-1 - m^{-1}\sum_{J=1}^m  n > -1,\] 
since $\widetilde{p}_J\leq\widetilde{p}\leq m$ and 
$|\alpha_1^J|>0$, integrating in $\rho$  we finally obtain 
\[
\int_{B(0,R)} \prod_{J=1}^{m} f_J(x) dx  \lesssim R^\delta \prod_{J=1}^{m} \| f_J \|_{L^{\widetilde{p}_J}(\R^{n-|\alpha_1^J|})},
\]
with $\widetilde\delta={n-\sum_{J=1}^m\widetilde{p}_J^{-1}(n-|\alpha^J_1|)}>0$. 
\end{proof}
\begin{remark}
By Theorem \ref{BLtheor_1}, inequality \eqref{localBL} also holds for all exponents $p_J\geq \widetilde{p}_J$ with the same argument 
as in Theorem \ref{local_th}. For $R>1$ this is also a consequence of the fact that 
$\sum_{J=1}^m (n-|\alpha_1^J|)p_J^{-1}\leq \sum_{J=1}^m(n-|\alpha_1^J|)\widetilde{p}_J^{-1}$
for $p_J\geq \widetilde{p}_J$. Indeed, when $R>1$, inequality \eqref{localBL} holds for all $\delta\geq\widetilde\delta$. 
\end{remark}
If we are in the balanced setting of Theorem \ref{expsym}, we also have the following result.
\begin{proposition}
With the hypotheses of Theorem \ref{expsym}, we have that inequality \eqref{localBL} holds with $\widetilde{p}_J=\widetilde p$ for all $J$, where $\widetilde p$ is given by formula $\eqref{exponent}$ and 
\[
\widetilde\delta=n-\widetilde{p}^{-1}(n-\widetilde{\alpha}_1)J_{\max}.
\]
Moreover $\widetilde\delta$ is sharp {in the sense that under our assumptions, inequality \eqref{localBL} is false in general for $\delta<\widetilde\delta$}. 
\end{proposition}
\begin{proof}
The expression for $\widetilde\delta$ follows by specializing Theorem \ref{local_th} to the setting of Theorem \ref{expsym}. 
In this setting, taking $\delta<\widetilde\delta$ is equivalent to choosing $p<\widetilde{p}$. Recalling that 
\[
\int_{B_n(0,R)} \prod_{J=1}^m f_J(x) dx =  \int_0^R \int_{\SP^{n-1}} \prod_{J=1}^{m} f_J(\rho x') d\sigma\rho^{n-1}d\rho
\]
and that by Theorem \ref{sharp_exp} for $p<\widetilde p$ there exist functions $f_J$, each $\mathcal{A}^J$-symmetric for $J=1,\dots,J_{\max}$, such that the 
inner integral diverges, we have that inequality \eqref{localBL} cannot hold for $\delta<\widetilde\delta$. 
\end{proof}

\footnotesize
\bibliographystyle{plain}
\bibliography{Bibliography}
\end{document}